\documentclass[a4paper]{amsart}
\usepackage{amsmath,amssymb,amsthm, url, color, lscape, float}
\usepackage[pdftex]{graphicx}

\restylefloat{table}

\title{Genus $2$ paramodular Eisenstein congruences}
\author{Dan Fretwell}
\thanks{Dan Fretwell, Heilbronn Institute for Mathematical Research, School of Mathematics, University of Bristol, U.K. Email: \url{daniel.fretwell@bristol.ac.uk}}
\date{}

\theoremstyle{plain}
\newtheorem{thm}{Theorem}[section]
\newtheorem{lem}[thm]{Lemma}
\newtheorem{cor}[thm]{Corollary}
\newtheorem{prop}[thm]{Proposition}
\newtheorem{conj}[thm]{Conjecture}

\theoremstyle{definition}
\newtheorem{define}[thm]{Definition}
\newtheorem{example}[thm]{Example}

\begin{document}
\maketitle

\begin{center}
\textbf{Abstract}
\end{center}

We investigate certain Eisenstein congruences, as predicted by Harder, for level $p$ paramodular forms of genus $2$. We use algebraic modular forms to generate new evidence for the conjecture. In doing this we see explicit computational algorithms that generate Hecke eigenvalues for such forms.

\section{Introduction}

Congruences between modular forms have been found and studied for many years. Perhaps the first interesting example is found in the work of Ramanujan. He studied in great detail the Fourier coefficients $\tau(n)$ of the discriminant function $\Delta(z) = q\prod_{n=1}^{\infty}(1-q^n)^{24}$ (where $q = e^{2\pi i z}$). The significance of $\Delta$ is that it is the unique normalized cusp form of weight $12$. 

Amongst Ramanujan's mysterious observations was a pretty congruence: \[\tau(n) \equiv \sigma_{11}(n) \bmod 691.\] Here $\sigma_{11}(n) = \sum_{d\mid n}d^{11}$ is a power divisor sum. Naturally one wishes to explain the appearance of the modulus $691$. The true incarnation of this is via the fact that the prime $691$ divides the numerator of the ``rational part" of $\zeta(12)$, i.e $\frac{\zeta(12)}{\pi^{12}}\in\mathbb{Q}$ (a quantity that appears in the Fourier coefficients of the Eisenstein series $E_{12}$).

Since the work of Ramanujan there have been many generalizations of his congruences. Indeed by looking for big enough primes dividing numerators of normalized zeta values one can provide similar congruences at level $1$ between cusp forms and Eisenstein series for other weights. In fact one can even give ``local origin" congruences between level $p$ cusp forms and level $1$ Eisenstein series by extending the divisibility criterion to include single Euler factors of $\zeta(s)$ rather than the global values of $\zeta(s)$ (see \cite{danneil} for results and examples).

There are also Eisenstein congruences predicted for Hecke eigenvalues of genus $2$ Siegel cusp forms. One particular type was conjectured to exist by Harder \cite{harder1}. There is only a small amount of evidence for this conjecture, the literature only contains examples at levels $1$ and $2$ (using methods specific to these levels). The conjecture is also far from being proved. Only one specific level $1$ example of the congruence has been proved (p.$386$ of \cite{chenevier2}).

In this paper we will see new evidence for a level $p$ version of Harder's conjecture for various small primes (including $p=2$ but not exclusively). The Siegel forms will be of paramodular type and the elliptic forms will be of $\Gamma_0(p)$ type. In doing this we will make use of Jacquet-Langlands style conjectures due to Ibukiyama.

\section{Harder's conjecture}

Given $k\geq 0$ and $N\geq 1$ let $S_{k}(\Gamma_0(N))$ denote the space of elliptic cusp forms for $\Gamma_0(N)$. Also for $j\geq 0$ let $S_{j,k}(K(N))$ denote the space of genus $2$, vector-valued Siegel cusp forms for the paramodular group of level $N$, taking values in the representation space $\text{Symm}^j(\mathbb{C}^2) \otimes \text{det}^k$ of $\text{GL}_2(\mathbb{C})$.

Given $f\in S_k(\Gamma_0(N))$ we let $\Lambda_{\text{alg}}(f,j+k) = \frac{\Lambda(f,j+k)}{\Omega}$, where $\Lambda(f,s)$ is the completed L-function attached to $f$ and $\Omega$ is a Deligne period attached to $f$. The choice of $\Omega$ is unique up to scaling by $\mathbb{Q}_f^{\times}$ but Harder shows how to construct a more canonical choice of $\Omega$ that is determined up to scaling by $\mathcal{O}_{\mathbb{Q}_f}^{\times}$ \cite{harder3}.

In this paper we consider the following paramodular version of Harder's conjecture (when $N=1$ this is the original conjecture found in \cite{harder1}).

\begin{conj}\label{V}

Let $j>0$, $k\geq 3$ and let $f\in S_{j+2k-2}^{\text{new}}(\Gamma_0(N))$ be a normalized Hecke eigenform with eigenvalues $a_n$. Suppose that $\text{ord}_\lambda(\Lambda_{alg}(f,j+k)) > 0$ for some prime $\lambda$ of $\mathbb{Q}_f$ lying above a rational prime $l > j+2k-2$ (with $l\nmid N$). 

Then there exists a Hecke eigenform $F\in S_{j,k}^{\text{new}}(K(N))$ with eigenvalues $b_n\in\mathbb{Q}_F$ such that \[b_q \equiv q^{k-2} + a_q + q^{j+k-1} \bmod \Lambda\] for all primes $q\nmid N$ (where $\Lambda$ is some prime lying above $\lambda$ in the compositum $\mathbb{Q}_f\mathbb{Q}_F$). 
\end{conj}

It should be noted that Harder's conjecture has still not been proved for level $1$ forms. However the specific example with $j=4, k=10$ and $l=41$ mentioned in Harder's paper has recently been proved in a paper by Chenevier and Lannes \cite{chenevier2}. The proof uses the Niemeier classification of $24$-dimensional lattices and is specific to this particular case.

Following the release of the level $1$ conjecture, Faber and Van der Geer were able to do computations when $\text{dim}(S_{j,k}(\text{Sp}_4(\mathbb{Z})))=1$. They have now exhausted such spaces and in each case have verified the congruence for a significant number of Hecke eigenvalues. Extra evidence for the case $j=2$ is given by Ghitza, Ryan and Sulon \cite{ghitza2}.

For the level $p$ conjecture a substantial amount of evidence has been provided by Bergstr\"om et al for level $2$ forms \cite{bergstrom}. Their methods are specific to this level. A small amount of evidence is known beyond level $2$. In particular a congruence has been found with $(j,k,p,l) = (0,3,61,43)$ by Anton Mellit (p.$99$ of \cite{harder3}).

In this paper we use the theory of algebraic modular forms to provide evidence for the conjecture at levels $p=2,3,5,7$. The methods discussed can be extended to work for other levels.

\section{Algebraic Modular Forms}
In general, it is quite tough to compute Hecke eigensystems for paramodular forms. Fortunately for a restricted set of levels there is a (conjectural) Jacquet-Langlands style correspondence for $\text{GSp}_4$ due to Ihara and Ibukiyama \cite{ibuk1}. 

Explicitly it is expected that there is a Hecke equivariant isomorphism between the spaces $S_{j,k}^{\text{new}}(K(p))$ and certain spaces of algebraic modular forms. Bearing this in mind we give the reader a brief overview of the general theory of such forms.

\subsection{The spaces $\mathcal{A}(G,K_f,V)$ of algebraic forms}

Let $G/\mathbb{Q}$ be a connected reductive group with the added condition that the Lie group $G(\mathbb{R})$ is connected and compact modulo center. Fix an open compact subgroup $K_f \subset G(\mathbb{A}_f)$. Also let $V$ be (the space of) a finite dimensional algebraic representation of $G$, defined over a number field $F$.

\begin{define}
The $F$-vector space of \textit{algebraic modular forms} of level $K_f$, weight $V$ for $G$ is: \[\mathcal{A}(G,K_f,V) \cong \{f: G(\mathbb{A}_f) \rightarrow V\,|\,f(\gamma g k) = \gamma f(g), \forall(\gamma,g,k)\in G(\mathbb{Q})\times G(\mathbb{A})\times K_f\}.\]
\end{define}

Fix a set of representatives $T = \{z_1,z_2,...,z_h\}\in G(\mathbb{A}_f)$ for $G(\mathbb{Q})\backslash G(\mathbb{A}_f)/ K_f$. There is a natrual embedding: \[\phi: \mathcal{A}(G,K_f,V) \longrightarrow V^h\]\[f \longmapsto (f(z_1),...,f(z_h)).\]

\begin{thm}\label{AHAH}
The map $\phi$ induces an isomorphism: \[\mathcal{A}(G,K_f,V) \cong \bigoplus_{m=1}^h V^{\Gamma_m},\] where $\Gamma_m = G(\mathbb{Q})\cap z_m K_f z_m^{-1}$ for each $m$.
\end{thm}

\begin{cor}\label{H}
The spaces $\mathcal{A}(G,K_f, V)$ are finite dimensional.
\end{cor}

A pleasing feature of the theory is that the groups $\Gamma_m$ are often finite. Gross gives many equivalent conditions for when this happens \cite{gross}. One such condition is the following.
 
\begin{prop}\label{ABC}
The groups $\Gamma_m$ are finite if and only if $G(\mathbb{Z})$ is finite.
\end{prop}

\subsection{Hecke Operators}

Let $u\in G(\mathbb{A}_f)$ and fix a decomposition $K_f u K_f = \coprod_{i=1}^r u_i K_f$. It is well known that finitely many representatives occur. Then $T_u$ acts on $f\in\mathcal{A}(G,K_f,V)$ via \[T_u(f)(g) := \sum_{i=1}^r f(gu_i), \qquad \forall g\in G(\mathbb{A}_f).\] It is easy to see that this is independent of the choice of representatives $u_i$ since they are determined up to right multiplication by $K_f$.

We wish to find the Hecke representatives $u_i$ explicitly and efficiently. To this end a useful observation can be made when the class number is one.

\begin{prop}\label{E}
If $h=1$ then we may choose Hecke representatives that lie in $G(\mathbb{Q})$.
\end{prop}

Finally we note that for $G$ satisfying Proposition \ref{ABC} there is a natural inner product on the space $\mathcal{A}(G,K,V)$. This is given in Gross' paper \cite{gross} but we shall give the rough details here.

\begin{lem}\label{MDN}
Let $G$ satisfy the property of Proposition \ref{ABC} and $V$ be a finite dimensional algebraic representation of $G$, defined over $\mathbb{Q}$. Then there exists a character $\mu : G \rightarrow \mathbb{G}_m$ and a positive definite symmetric bilinear form $\langle,\rangle: V\times V\rightarrow\mathbb{Q}$ such that: \[\langle\gamma u, \gamma v\rangle = \mu(\gamma)\langle u,v\rangle\] for all $\gamma\in G(\mathbb{Q})$.
\end{lem}

Taking adelic points we have a character $\mu':G(\mathbb{A}) \rightarrow \mathbb{A}^{\times}$. Let $\mu_{\mathbb{A}} = f\circ \mu'$, where $f: \mathbb{A}^{\times} \rightarrow \mathbb{Q}^{\times}$ is the natural projection map coming from the decomposition $\mathbb{A}^{\times} = \mathbb{Q}^{\times}\mathbb{R}^+\hat{\mathbb{Z}}^{\times}$.

\begin{prop}\label{AZA}
Let $G$ satisfy the property of Proposition \ref{ABC}. Then $\mathcal{A}(G,K,V)$ has a natural inner product given by: \[\langle f,g\rangle = \sum_{m=1}^h \frac{1}{|\Gamma_m|\mu_{\mathbb{A}}(z_m)}\langle f(z_m), g(z_m) \rangle.\]
\end{prop}

\subsection{Trace of Hecke operators}

The underlying representation $V$ of $G$ is typically big in dimension and so the action of Hecke operators is, although explicit, quite tough to compute. Fortunately, there is a simple trace formula for Hecke operators on spaces of algebraic modular forms. The details of the formula can be found in \cite{dum2} but we give brief details here.

Note that $G(\mathbb{A}_f)$ acts on the set $Z$ on the left by setting $w\cdot z_i = z_j$ if and only if $G(\mathbb{Q})(wz_i)K_f = G(\mathbb{Q})z_jK_f$.
For each $m=1,2,...,h$ we consider the set $S_m = \{i\,|\, u_i\cdot z_m = z_m\}$. Next for each $i\in S_m$ choose elements $k_{m,i}\in K_f$ and $\gamma_{m,i}\in G(\mathbb{Q})$ such that $\gamma_{m,i}^{-1} u_i z_m k_{m,i} = z_m$.

Let $\chi_V$ denote the character of the representation of $G(\mathbb{Q})$ on $V$. Then the trace formula is as follows.

\begin{thm}(Dummigan)\label{K} \[\text{tr}(T_u) = \sum_{m=1}^{h}\frac{1}{|\Gamma_m|}\sum_{\gamma\in\Gamma_m ,\\ i\in S_m} \chi_V(\gamma_{m,i}\gamma).\]

More generally: \[\text{tr}(T_u^d) = \sum_{m=1}^{h}\frac{1}{|\Gamma_m|}\sum_{\gamma\in\Gamma_m, \\ (i_n)\in S_m^d} \chi_V\left(\left(\prod_{n=1}^d \gamma_{m,i_n}\right)\gamma\right).\]
\end{thm}

Letting $u=\text{id}$ we recover the following.

\begin{cor}\label{I}
\[\text{dim}(\mathcal{A}(G,K_f,V)) = \sum_{m=1}^{h}\frac{1}{|\Gamma_m|}\sum_{\gamma\in\Gamma_m} \chi_V(\gamma).\]
\end{cor}

When $h=1$ the situation becomes much simpler. In this case we may choose $z_1 = \text{id}$ and $\gamma_{1,i} = u_i\in G(\mathbb{Q})$ for each $i$ (this is possible by Corollary \ref{E}).

\begin{cor}\label{J}
If $h=1$ then we have \[\text{tr}(T_u) = \frac{1}{|\Gamma|}\sum_{\gamma\in\Gamma, 1\leq i\leq r}\chi_V(u_i \gamma),\] where $\Gamma = G(\mathbb{Q}) \cap K_f$.
\end{cor}

The trace formula was introduced to test a U$(2,2)$ analogue of Harder's conjecture. In this paper we will use it to test the level $p$ paramodular version of Harder's conjecture given by Conjecture \ref{V}.

\section{Eichler and Ibukiyama correspondences}

\subsection{Eichler's correspondence}

From now on $D$ will denote a quaternion algebra over $\mathbb{Q}$ ramified at $\{p,\infty\}$ (for a fixed prime $p$) and $\mathcal{O}$ will be a fixed maximal order. Since $D$ is definite, we have that $D_{\infty}^{\times} = D^{\times}\otimes\mathbb{R} \cong \mathbb{H}^{\times}$ is compact modulo center (and is also connected). Thus we may consider algebraic modular forms for the group $G = D^{\times}$.

Also note that in this case each $\Gamma_m$ will be finite since $D^{\times}(\mathbb{Z}) = \mathcal{O}^{\times}$ is finite.

Let $D_q := D \otimes \mathbb{Q}_q$ be the local component at prime $q$ (no restriction on $q$) and let $D_{\mathbb{A}_f}$ be the restricted direct product of the $D_q$'s with respect to the local maximal orders $\mathcal{O}_q := \mathcal{O} \otimes \mathbb{Z}_q$.

Note that if $q\neq p$ then $D_q^{\times} \cong (M_2(\mathbb{Q}_q))^{\times} = \text{GL}_2(\mathbb{Q}_q)$. Thus locally away from the ramified prime, $D^{\times}$ behaves like $\text{GL}_2$.

In fact more is true. It is the case that the reductive groups $D^{\times}$ and GL$_2$ are inner forms of each other. So by the principle of Langlands functoriality we expect a transfer of automorphic forms between $D^{\times}$ and GL$_2$. Eichler gives an explicit description of this transfer.

Let $V_n = \text{Symm}^n(\mathbb{C}^2)$ (for $n\geq 0$). Then $V_n$ gives a well defined representation of $\text{SU}(2)/\{\pm I\}$ if and only if $n$ is even. Thus we get a well defined action on $V_n$ by $D^{\times}$ via: \[D^{\times} \hookrightarrow \mathbb{H}^{\times} \longrightarrow \mathbb{H}^{\times}/\mathbb{R}^{\times} \cong \text{SU}(2)/\{\pm I\}.\]

Take $U = \prod_{q}\mathcal{O}_q^{\times}$. This is an open compact subgroup of $D_{\mathbb{A}_f}^{\times}$.

\begin{thm}(Eichler)
Let $k>2$. Then there is a Hecke equivariant isomorphism: \[S_{k}^{\text{new}}(\Gamma_0(p)) \cong \mathcal{A}(D^{\times},U, V_{k-2}).\]

For $k=2$ the above holds if on the right we quotient out by the space of constant functions.
\end{thm}

It remains to describe how the Hecke operators transfer over the isomorphism. Fix a prime $q\neq p$. Choose $u\in D_{\mathbb{A}_f}^{\times}$ such that $\psi(u_q) = \text{diag}(1,p)$ and is the identity at all other places. The corresponding Hecke operator $T_{u,q}$ corresponds to the classical $T_q$ operator under Eichler's correspondence.

\subsection{Ibukiyama's correspondence}

Ibukiyama's correspondence is a (conjectural) generalisation of Eichler's correspondence to Siegel modular forms. The details can be found in \cite{ibuk1} but we explain the main ideas.

Given the setup in the previous subsection, consider the unitary similitude group: \[\text{GU}_{n}(D) = \{g \in \text{M}_n(D)\,|\,g\bar{g}^{T} = \mu(g) I, \, \mu(g)\in\mathbb{Q}^{\times}\}.\] Here $\bar{g}$ means componentwise application of the standard involution of $D$. This group is the similitude group of the standard Hermitian form on $D^n$.

\begin{thm}\label{L}
For any field $K$ there exists a similitude-preserving isomorphism $\text{GU}_2(M_2(K)) \cong \text{GSp}_4(K)$.
\end{thm}

\begin{proof}
Conjugation by the matrix $M = \text{diag}(1,A,1)\in M_4(K)$, where $A = \left(\begin{array}{cc}0 & 1\\1 & 0\end{array}\right)$ gives such an isomorphism.
\end{proof}

One consequence of this is that the group GU$_2(D)$ behaves like GSp$_4$ locally away from the ramified prime. It is indeed true that these groups are also inner forms of each other.

A simple argument also shows that $\text{GU}_2(\mathbb{H})/Z(\text{GU}_2(\mathbb{H})) \cong \text{USp}(4)/\{\pm I\}$. Thus GU$_2(D_{\infty})$ is compact modulo center and connected. Thus we may consider algebaric modular forms for this group. Once again we are guaranteed that the $\Gamma_m$ groups are finite by the following.

\begin{lem}\label{YUN}
GU$_2(\mathcal{O})= \{\gamma\in\text{GU}_2(D)\cap \text{M}_2(\mathcal{O})\,|\,\mu(\gamma)\in\mathbb{Z}^{\times}\}$ is finite.
\end{lem}

\begin{proof}
Solving the equations gives: \[\text{GU}_2(\mathcal{O}) = \left\{\left(\begin{array}{cc}\alpha & 0\\ 0 & \beta\end{array}\right), \left(\begin{array}{cc} 0 & \alpha\\ \beta & 0\end{array}\right)\,\Bigg|\,\alpha,\beta\in\mathcal{O}^{\times}\right\}.\]
\end{proof}

One consequence of Theorem \ref{L} is that GU$_2(D_q) \cong \text{GSp}_4(\mathbb{Q}_q)$for all $q\neq p$.

\begin{prop}\label{P}
For any $q\neq p$ there exists a similitude-preserving isomorphism $\psi: \text{GU}_{2}(D_q) \rightarrow \text{GSp}_4(\mathbb{Q}_q)$ that preserves integrality: \[\psi(\text{GU}_2(D_q)\cap M_2(\mathcal{O}_q)) = \text{GSp}_4(\mathbb{Q}_q)\cap M_4(\mathbb{Z}_q).\]
\end{prop}

\begin{proof}
Choose an isomorphism of quaternion algebras $D_q \cong M_2(\mathbb{Q}_q)$ that preserves the norm, trace and integrality. This induces an isomorphism with the required properties: \[\text{GU}_2(D_q) \cong \text{GU}_2(M_2(\mathbb{Q}_q)) \cong \text{GSp}_4(\mathbb{Q}_q).\]\end{proof}

Let $V_{j,k-3}$ be the irreducible representation of USp$(4)$ with Young diagram parameters $(j+k-3,k-3)$. This gives a well defined representation of USp$(4)/\{\pm I\}$ if and only if $j$ is even. Thus GU$_2(D)$ acts on this via: \[ \text{GU}_2(D) \hookrightarrow \text{GU}_2(\mathbb{H}) \longrightarrow \text{GU}_2(\mathbb{H})/Z(\text{GU}_2(\mathbb{H})) \cong \text{USp}(4)/\{\pm I\}.\]

The groups $\text{GU}(D)$ and $\text{GSp}_4$ are inner forms. Thus (as with Eichler) one expects a transfer of automorphic forms. The following is found in Ibukiyama's paper \cite{ibuk1}.

\begin{conj}(Ibukiyama)
Let $j\geq 0$ be an even integer and $k\geq 3$. Supppose $(j,k) \neq (0,3)$. Then there there is a Hecke equivariant isomorphism:
\[S_{j,k}^{\text{new}}(\Gamma_0(p)) \longrightarrow \mathcal{A}^{\text{new}}(GU_2(D),U_1,V_{j,k-3})\] \[S_{j,k}^{\text{new}}(K(p)) \longrightarrow \mathcal{A}^{\text{new}}(\text{GU}_2(D),U_2, V_{j,k-3}),\] where $U_1, U_2, V_{j,k-3}$ are to be defined. 

If $(j,k) = (0,3)$ then we also get an isomorphism after taking the quotient by the constant functions on the right.
\end{conj}

Since our eventual goal is to study Harder's conjecture for paramodular forms we will neglect the first of these isomorphisms. However, it will turn out that the open compact subgroup $U_1$ will prove useful in later calculations.

\subsubsection{The levels $U_1$ and $U_2$.}

In Eichler's correspondence, the ``level $1$" open compact subgroup $U = \prod_q \mathcal{O}_q^{\times} \subset D_{\mathbb{A}_f}^{\times}$ can be viewed as Stab$_{D_{\mathbb{A}_f}^{\times}}(\mathcal{O})$ under an action defined by right multiplication. Similarly, one can produce open compact subgroups Stab$_{\text{GU}_2(D_{\mathbb{A}_f})}(L)\subseteq \text{GU}_2(D_{\mathbb{A}_f})$, where $L$ is a left $\mathcal{O}$-lattice of rank $2$ in $D^2$ (a free left $\mathcal{O}$-module of rank $2$).

A left $\mathcal{O}$-lattice $L\subseteq D^2$ gives rise to a left $\mathcal{O}_q$-lattice $L_q = L\otimes \mathbb{Z}_q \subseteq D_q^2$ for each prime $q$. A result of Shimura tells us the possibilities for $L_q$ (see \cite{shimura}).

\begin{thm}
Let $D$ be a quaternion algebra over $\mathbb{Q}$. If $D$ is split at $q$ then $L_q$ is right GU$_2(D_q)$ equivalent to $\mathcal{O}_q^2$. If $D$ is ramified at $q$ then there are exactly two possibilities for $L_q$, up to right GU$_2(D_q)$ equivalence (one being $\mathcal{O}_q^2$).
\end{thm}

When $D$ is ramified at $\{p,\infty\}$ it is clear from this result that there are only two possibilities for $L$, up to local equivalence.

\begin{define}
Let $D$ be ramified at $p,\infty$ for some prime $p$:
\begin{itemize}
\item{If $L_p$ is locally equivalent to $\mathcal{O}_p^2$ for all $q$ then we say that $L$ lies in the principal genus.} 
\item{If $L_p$ is locally inequivalent to $\mathcal{O}_p^2$ then we say that $L$ lies in the non-principal genus.\qed}
\end{itemize}
\end{define}

Given $L$, results of Ibukiyama \cite{ibuk4} allow us to write $L = \mathcal{O}^2 g$ for some $g\in\text{GL}_2(D)$ and determine the genus of $L$ based on $g$.

\begin{thm}\label{AQQ}
\begin{itemize}
\item{$L$ lies in the principal genus if and only if $g\overline{g}^T = mx$ for some positive $m\in\mathbb{Q}$ and some $x\in\text{GL}_n(\mathcal{O})$ such that $x = \overline{x}^T$ and such that $x$ is positive definite, i.e. $yx\overline{y}^T > 0$ for all $y\in D^n$ with $y\neq 0$.}
\item{$L$ lies in the non-principal genus if and only if $g\overline{g}^T = m\left(\begin{array}{cc}ps & r\\ \overline{r} & pt\end{array}\right) $ where $m\in\mathbb{Q}$ is positive, $s,t\in\mathbb{N}, r\in \mathcal{O}$ lies in the two sided ideal of $\mathcal{O}$ above $p$ and is such that $p^2st - N(r) = p$ (so that the matrix on the right has determinant $p$).}
\end{itemize}
\end{thm}

The lattice $\mathcal{O}^2$ is clearly in the principal genus and corresponds to the choice $g=I$. Alternatively fix a choice of $g$ such that $\mathcal{O}^2 g$ is in the non-principal genus. Let $U_1,U_2$ respectively denote the corresponding open compact subgroups of $\text{GU}_2(D_{\mathbb{A}_f})$ (as described above).

\subsubsection{Hecke operators}

The transfer of Hecke operators in Ibukiyama's correspondence is similar to the Eichler correspondence but has subtle differences. Fix a prime $q\neq p$ and let $M_q = \text{diag}(1,1,q,q)\in\text{GSp}_4(\mathbb{Q}_q)$. Fixing an isomorphism as in Proposition \ref{P} we may choose $v_q\in\text{GU}_2(D_q)$ such that $v_q \mapsto M_q$. Since $q\neq p$ we know that $\mathcal{O}_q^2 g_q = \mathcal{O}_q^2 h_q$ for some $h_q\in \text{GU}_2(D_q)$ (where $g_q$ is the image of $g$ under the natural embedding $\text{GU}_2(D) \rightarrow \text{GU}_2(D_q)$.  $u_q\in\text{GU}_2(D_q)$ given by $u_q = h_q v_q h_q^{-1}$. 

Let $u\in\text{GU}_2(D_{\mathbb{A}_f})$ have $u_q = h_q v_q h_q^{-1}$ as the component at $q$ and have identity component elsewhere.

\begin{define}\label{Q}
For the above choice of $u$, the corresponding Hecke operator on $\mathcal{A}^{\text{new}}(\text{GU}_2(D),U_2, V_{j,k-3})$ will be called $T_{u,q}$.\qed
\end{define} 

Under Ibukiyama's correspondence it is predicted that $T_{u,q}$ corresponds to the classical $T_q$ operator acting on $S_{j,k}^{\text{new}}(K(p))$).

\subsubsection{The new subspace}

Our final task in defining Ibukiyama's correspondence is to explain what is meant by the new subspace $\mathcal{A}^{\text{new}}(\text{GU}_2(D),U_2,V_{j,k-3})$. We will not go into too much detail but will refer the reader to Ibukiyama's papers \cite{ibuk3}, \cite{ibuk5}.

Let $G = D^{\times}\times\text{GU}_2(D)$. Then we have an open compact subgroup $U' = U\times U_2$ and finite dimensional representations $W_{j,k-3} := V_{j}\otimes V_{j,k-3}$ of $G(\mathbb{A}_f)$.

We start with the decomposition: \[\mathcal{A}(G,U',W_{j,k-3}) \cong \mathcal{A}(D^{\times},U,V_j)\otimes \mathcal{A}(\text{GU}_2(D),U_2,V_{j,k-3}).\] Ibukiyama takes $F\in\mathcal{A}(G,U',W_{j,k-3})$. If $F$ is an eigenform then $F = F_1\otimes F_2$ for eigenforms $F_1,F_2$. He then associates an explicit theta series $\theta_F$ to $F$. This is an elliptic modular form for $SL_2(\mathbb{Z})$ of weight $j+2k-2$ (if $j+2k-6\neq 0$ then it is a cusp form). It is known that $\theta_F$ is an eigenform for all Hecke operators if and only if $\theta_F \neq 0$.

\begin{define}
The subspace of old forms $A_{j,k-3}^{\text{old}}(D) \subseteq A_{j,k-3}(D)$ is generated by the eigenforms $F_2$ such that there exists an eigenform $F_1$ satisfying $\theta_{F_1 \otimes F_2} \neq 0$.

The subspace of new forms $A_{j,k-3}^{\text{new}}(D)$ is the orthogonal complement of the old space with respect to the inner product in Proposition \ref{AZA}.\qed
\end{define}

It should be noted that by Eichler's correspondence $F_1$ can be viewed as an elliptic modular form for $\Gamma_0(p)$ of weight $j+2$. Further it will be a new cusp form precisely when $j>0$. Thus computationally it is not difficult to find the new and old subspaces.

\section{Finding evidence for Harder's conjecture}

Now that we have linked spaces of Siegel modular forms $S_{j,k}^{\text{new}}(K(p))$ with spaces of algebraic modular forms $A_{j,k-3}^{\text{new}}(D) = \mathcal{A}^{\text{new}}(\text{GU}_2(D),U_2, V_{j,k-3})$, we can begin to generate evidence for Harder's conjecture.

\subsection{Brief plan of the strategy}

In this paper we will deal with cases where $h=1$ and dim$(A_{j,k-3}^{\text{new}}(D)) = 1$.

\textbf{Strategy}

\begin{enumerate}
\item{Find all primes $p$ such that $h=1$.}
\item{For each such $p$ calculate $\Gamma^{(2)} = \text{GU}_2(D)\cap U_2$.}
\item{Using Corollary \ref{I} find all $j,k$ such that $\text{dim}(A_{j,k}^{\text{new}}(D))=1$.}
\item{For each pair $(j,k)$ look in the space of elliptic forms $S_{j+2k-2}^{\text{new}}(\Gamma_0(p))$ for normalized eigenforms $f$ which have a ``large prime" dividing $\Lambda_{\text{alg}}(f,j+k)\in\mathbb{Q}_f$.}
\item{Find the Hecke representatives for the $T_{u,q}$ operator at a chosen prime $q$.}
\item{Use the trace formula to find tr$(T_{u,q})$ for $T_q$ acting on $A_{j,k-3}(D)$.}
\item{Subtract off the trace contribution of $T_{u,q}$ acting on $A_{j,k-3}^{\text{old}}(D)$ in order to get the trace of the action on $A_{j,k-3}^{\text{new}}(D)$. Since dim$(A_{j,k-3}^{\text{new}}(D))=1$ this trace should be exactly the Hecke eigenvalue of a new paramodular eigenform by Ibukiyama's conjecture.}
\item{Check that Harder's congruence holds.}
\end{enumerate}

The above strategy can be modified to work for the case $\text{dim}(A_{j,k-3}^{\text{new}}(D)) = d > 1$ but one must compute $\text{tr}(T_{u,q}^t)$ for $1\leq t \leq d$.

\subsection{Finding $\Gamma^{(2)}$.}

For $\theta\in\mathbb{Q}^{\times}$ consider the subset: \[\text{GU}_{n}(D)_{\theta} = \{\gamma\in\text{GU}_2(D)\,|\,\mu(\gamma) = \theta\},\] In particular let SU$_2(D) := \text{GU}_2(D)_1$.

\begin{thm}
The group $\Gamma^{(2)}$ consists of the following set of matrices: \[\Gamma^{(2)} =\text{SU}_{2}(D)\cap g^{-1}\text{GL}_2(\mathcal{O})g\]
\end{thm}

\begin{proof}
We know that: \[\Gamma^{(2)} = \text{Stab}_{\text{GU}_2(D)}(\mathcal{O}^2 g) = \text{GU}_2(D) \cap \text{Stab}_{\text{GL}_2(D)}(\mathcal{O}^2 g)\] \[=  \text{GU}_2(D) \cap g^{-1}Sg = \text{GU}_{2}(D)\cap g^{-1}\text{GL}_2(\mathcal{O})g.\] A simple calculation shows that any such matrix has similitude $1$. 
\end{proof}

Recall also the open compact subgroup $U_1 = \text{Stab}_{\text{GU}_2(\mathbb{A}_f)}(\mathcal{O}^2) \subset \text{GU}_2(D_{\mathbb{A}_f})$. This is the stabilizer of a left $\mathcal{O}$-lattice lying in the principal genus.  

In this case the analogue of the group $\Gamma^{(2)}$ is the group $\Gamma^{(1)} = \text{GU}_2(D) \cap U_1$. We can employ identical arguments to the above to show the following: 

\begin{lem}
\[\Gamma^{(1)} = \text{SU}_2(D) \cap \text{GL}_2(\mathcal{O}) = \text{GU}_2(\mathcal{O}).\]
\end{lem} 

We already have an explicit description of $\Gamma^{(1)}$ (see Lemma \ref{YUN}). Computationally it is not straight forward to find the elements of $\Gamma^{(2)}$ due to the non-integrality of the entries of such matrices.

For $\theta\in\mathbb{Q}^{\times}$ consider the sets \[Y_{\theta} = \text{GU}_{2}(D)_{\theta}\cap g^{-1} M_2(\mathcal{O})^{\times}g\] and \[W_{\theta} = \{\nu\in M_2(\mathcal{O})^{\times}\,|\,\nu A \overline{\nu}^{T} = \theta A\},\] where $M_2(\mathcal{O})^{\times} = \text{GL}_2(D)\cap M_2(\mathcal{O})$ and $A = g\bar{g}^T$.

Then in particular $Y_1 = \Gamma^{(2)}$. Later the sets $Y_q$ for prime $q\neq p$ will appear when finding Hecke representatives.

\begin{prop}\label{O}
For each $\theta\in\mathbb{Q}^{\times}$ conjugation by $g$ gives a bijection: \[\Phi_{\theta}: Y_{\theta} \longrightarrow W_{\theta}.\]
\end{prop}

To calculate the sets $W_{\theta}$ we diagonalize $A$. Choose a matrix $P\in\text{GL}_2(D)$ such that $PA\overline{P}^{T} = B$ where $B\in M_2(D)$ is a diagonal matrix.

\begin{prop}
For each $\theta\in\mathbb{Q}^{\times}$ conjugation by $P$ gives a bijection \[W_{\theta} \longrightarrow Z_{\theta}:=\{\eta\in P\,\text{M}_2(\mathcal{O})^{\times}\, P^{-1}\,|\,\eta B\overline{\eta}^{T} = \theta B\}.\]
\end{prop}

If we make an appropriate choice of $g$ and $P$ then we can diagonalize $A$ in such a way as to preserve one integral entry in $P\nu P^{-1}$.

\begin{lem}\label{T}
Suppose we can choose $\lambda,\mu\in\mathcal{O}$ such that $N(\lambda)= p-1$, $N(\mu) = p$ and $\text{tr}(r) = 0$ (where $r = \lambda\overline{\mu}$). Then \[g_{\lambda,\mu} := \left(\begin{array}{cc} 1 & \lambda\\ 0 & \mu\end{array}\right)\quad\text{and}\quad P_{\lambda,\mu} = \left(\begin{array}{cc}1 & \frac{\overline{r}}{p}\\ 0 & 1\end{array}\right)\] are valid choices for $g$ and $P$.

Further $P_{\lambda,\mu}^{-1} = \overline{P_{\lambda,\mu}}$.
\end{lem}

\begin{proof}
A simple calculation shows that: \[A_{\lambda,\mu} = \left(\begin{array}{cc} 1 & \lambda\\ 0 & \mu\end{array}\right)\left(\begin{array}{cc} 1 & 0\\ \overline{\lambda} & \overline{\mu}\end{array}\right) = \left(\begin{array}{cc} 1 + N(\lambda) & \lambda\overline{\mu}\\ \mu\overline{\lambda} & N(\mu)\end{array}\right) = \left(\begin{array}{cc}p & r\\ \overline{r} & p\end{array}\right),\] and also that det$(A_{\lambda,\mu}) = p^2 - N(r) = p^2 - p(p-1) = p$ as required.

To prove the second claim we note that $r^2 = -p(p-1)$ by the Cayley-Hamilton theorem (since $\text{tr}(r) = 0$ and $N(r) = p(p-1)$). Then \[P_{\lambda,\mu}A\overline{P_{\lambda,\mu}}^T = \left(\begin{array}{cc} 1 & \frac{\overline{r}}{p}\\ 0 & 1\end{array}\right)\left(\begin{array}{cc} p & r\\ \overline{r} & p\end{array}\right)\left(\begin{array}{cc} 1 & 0\\ \frac{r}{p} & 1\end{array}\right) = \left(\begin{array}{cc}p + \frac{r^2}{p} + \frac{\overline{r}}{p}(\text{tr}(r)) & \text{tr}(r)\\ \text{tr}(r) & p\end{array}\right)\] and so $P_{\lambda,\mu}A\overline{P_{\lambda,\mu}}^T =\text{diag}(1,p)$.

The final claim follows from the fact that $P_{\lambda,\mu}\overline{P_{\lambda,\mu}} = I$ (which again uses the fact that $\text{tr}(r) = 0$).
\end{proof}

It is in fact always possible to find \textbf{some} maximal order $\mathcal{O}$ of $D$ where such $\lambda,\mu$ exist. For proof of this I refer to an online discussion with John Voight \cite{url}, of which the author is grateful. We fix such a choice from now on.

\begin{cor}\label{N}
Let $\nu\in\text{M}_2(\mathcal{O})$. Then the bottom left entries of $\nu$ and $P_{\lambda,\mu}\nu\overline{P_{\lambda,\mu}}$ are equal (in particular this entry remains in $\mathcal{O}$).
\end{cor}

\begin{proof}
Let $\nu = \left(\begin{array}{cc}\alpha & \beta \\ \gamma & \delta\end{array}\right)$ with $\alpha,\beta,\gamma,\delta\in\mathcal{O}$. Then a simple calculation shows that \[P_{\lambda,\mu}\nu\overline{P_{\lambda,\mu}} = \left(\begin{array}{cc} \alpha + \frac{\overline{r}\gamma}{p} & (\frac{\alpha r}{p}+\beta) + \frac{\overline{r}}{p}(\frac{\gamma r}{p} + \delta)\\ \gamma & \frac{\gamma r}{p} + \delta\end{array}\right).\]
\end{proof}

The matrix $\eta = \left(\begin{array}{cc}x & y\\ z & w\end{array}\right)\in M_2(D)$ belongs to $Z_{\theta}$ if and only if \[\eta\left(\begin{array}{cc}1 & 0\\ 0 & p\end{array}\right)\overline{\eta}^T = \theta  \left(\begin{array}{cc}1 & 0\\ 0 & p\end{array}\right)\] Equivalently \[N(x) + pN(y) = \theta\]\[N(z) + pN(w) = \theta p\]\[x\overline{z} + py\overline{w} = 0.\]

Clearly these equations can have no solutions for $\theta < 0$ and so we only consider $\theta\geq 0$. 

A quick calculation shows that $N(x) = N(w)$ and $N(z) = p^2N(y)$ (a fact we will use soon). 

\begin{cor}\label{ABCZ}
Let $\theta\geq 0$. Then $W_{\theta}$ consists of all matrices $\nu = \left(\begin{array}{cc}\alpha & \beta\\ \gamma & \delta\end{array}\right)\in M_2(\mathcal{O})^{\times}$ such that: \[pN(p\alpha + \overline{r}\gamma) + N(p(\alpha r + p\beta) + \overline{r}(\gamma r + p\delta)) = \theta p^3\]\[pN(\gamma) + N(\gamma r +p\delta) = \theta p^2\]\[p\alpha\overline{\gamma} + (\alpha r+p\beta)(\overline{\gamma r + p\delta}) = -\theta p \overline{r}.\] 
\end{cor}

The following algorithm allows us to comute $W_{\theta}$ for $\theta\in \mathbb{N}$. Denote by $X_i$ the subset of $\mathcal{O}$ consisting of norm $i$ elements.

\textbf{Algorithm 1}

\textbf{Step 0:} Set $j:=0$. For each integer $0\leq i\leq \theta p$, generate the norm lists $X_i, X_{p(\theta p - i)}, X_{p^2 i}$.

\textbf{Step 1:} For each pair of elements $(\gamma,\gamma')\in X_j\times X_{p(\theta p-j)}$ check whether the element $\delta := \frac{\gamma'-\gamma r}{p}\in\mathcal{O}$. 

\textbf{Step 2:} For each putative $\gamma\in X_j$ from Step $1$ find all elements $\gamma''\in X_{p(\theta p - j)}$ such that the element $\alpha := \frac{\gamma'' - \overline{r}\gamma}{p}\in\mathcal{O}$.

\textbf{Step 3:} For each putative triple $(\alpha,\gamma,\delta)$ from Step $2$ and each $\gamma'''\in X_{p^2 j}$ test whether the element $\beta:= \frac{\gamma''' - (\overline{r}(\gamma r + p\delta) + p\alpha r)}{p^2}\in\mathcal{O}$.

\textbf{Step 4:} Check that the entries of each putative tuple from Step $3$ satisfies the third equation of Corollary \ref{ABCZ}.

\textbf{Step 5:} Set $j:=j+1$ and repeat steps 1-4 until $j>\theta p$.

Of course once the elements of $W_{\theta}$ have been found it is straight forward to generate the elements of $Y_{\theta}$ by inverting the bijection $\Phi_{\theta}$ in Proposition \ref{O}.

It should be noted that if we run this algorithm for $p=2$ with the following choices \[D = \left(\frac{-1,-1}{\mathbb{Q}}\right)\]\[\mathcal{O} = \mathbb{Z}\oplus\mathbb{Z}i\oplus\mathbb{Z}j\oplus\mathbb{Z}\frac{1+i+j+k}{2}\]\[\lambda = -1\]\[\mu = i-k\]\[\theta = 1\] then we get exactly the same elements for $Y_1 = \Gamma^{(2)}$ as Ibukiyama does on p.$592$ of \cite{ibuk1}.

\subsection{Finding $h$}

We can use mass formulae to get information on class numbers $h_1$ and $h_2$ for $U_1$ and $U_2$. 

Define the mass of open compact $U\subset \text{GU}_2(D_{\mathbb{A}_f})$ as follows: \[M(U) := \sum_{m=1}^{h}\frac{1}{|\Gamma_m|},\] where $\Gamma_m = \text{GU}_2(D) \cap z_m U z_m^{-1}$ for representatives $z_1, z_2, ..., z_m\in \text{GU}_2(D_{\mathbb{A}_f})$ of $\text{GU}_2(D)\backslash \text{GU}_2(D_{\mathbb{A}_f})/ U$.

Ibukiyama provides the following formulae for $M(U_1)$ and $M(U_2)$ in \cite{ibuk1}.

\begin{thm}
If $D$ is ramified at $p$ and $\infty$ then: \[M(U_1) = \frac{(p-1)(p^2+1)}{5760},\]\[M(U_2) = \frac{p^2 - 1}{5760}.\]
\end{thm}

This formula is analogous to the Eichler mass formula and is also a special case of the mass formula of Gan, Hanke and Yu \cite{gan}.

\begin{prop}
$h_1=1$ if and only if $|\Gamma^{(1)}| = \frac{5760}{(p-1)(p^2+1)}$. Similarly $h_2=1$ if and only if $|\Gamma^{(2)}| = \frac{5760}{p^2-1}$.
\end{prop}

\begin{cor}
$h_1 = 1$ if and only if $p=2,3$. Similarly $h_2=1$ if and only if $p=2,3,5,7,11$.
\end{cor}

\begin{proof}
A quick calculation shows that the only primes to satisfy $\frac{5760}{(p-1)(p^2+1)}\in\mathbb{N}$ are $p=2,3$. Recall $|\Gamma^{(1)}| = 2|\mathcal{O}^{\times}|^2$. For $p=2,3$ we have $|\mathcal{O}^{\times}| = 24, 12$ respectively and one checks that both values satisfy the equation.

The primes satisfying $\frac{5760}{p^2-1}\in\mathbb{N}$ are $p=2,3,5,7,11,17,19,31$. Using Algorithm $1$ one finds that $|\Gamma^{(2)}| = \frac{5760}{p^2-1}$ for the cases $p=2,3,5,7,11$.
\end{proof}

Ibukiyama and Hashimoto have produced formulae in \cite{hash1} and \cite{hash2} that give the values of $h_1$ and $h_2$ for any ramified prime. Their formulae agree with this result.

\subsection{Finding the Hecke representatives}

Now that we have found an algorithm to generate the elements of $\Gamma^{(2)}$ we consider the same question for the Hecke representatives for the $T_{u,q}$ operator on $A_{j,k-3}(D)$ (where $q\neq p$ is a fixed prime).

\begin{prop}\label{R}
Let $D$ be a quaternion algebra over $\mathbb{Q}$ ramified at $p,\infty$ for some $p\in\{2,3,5,7,11\}$. Suppose $u\in\text{GU}_2(D_{\mathbb{A}_f})$ is chosen as in Definition \ref{Q}. Then \[U_2 u U_2 = \coprod_{[x_i]\in Y_q / \Gamma^{(2)}} x_i U_2.\]
\end{prop}

\begin{proof}
Consider an arbitrary decomposition: \[U_2 u U_2 = \coprod x_i U_2.\] By Proposition \ref{E} we may take $x_i\in \text{GU}_2(D)$ for each $i$. For the rest of the proof we embed $\text{GU}_2(D) \hookrightarrow \text{GU}_2(D_{\mathbb{A}_f})$ diagonally.

Note that for any prime $l\neq q$ we have \[U_{2,l} u_l U_{2,l} = U_{2,l} = \text{Stab}_{\text{GU}_2(D_l)}(\mathcal{O}_l^2 g_l) =\text{GU}_2(D_l) \cap g_l^{-1}\text{GL}_2(\mathcal{O}_l) g_l\] Thus $x_i\in  \text{GU}_2(D_l) \cap g_l^{-1} M_2(\mathcal{O}_l)^{\times} g_l$ and $\mu(x_i) \in\mathbb{Z}_l^{\times}$ for all $i$.

To study the behaviour locally at $q$ we fix a choice of $h_q\in\text{GU}_2(D_q)$ such that $\mathcal{O}_q^2 g_q = \mathcal{O}_q^2 h_q$ (which is possible since $\mathcal{O}_q^2 g_q$ is locally equivalent to $\mathcal{O}_q^2$). Note that $h_q g_q^{-1}\in\text{GL}_2(\mathcal{O}_q)$ so that $h_q = k_q g_q$ for some $k_q\in\text{GL}_2(\mathcal{O}_q)\subseteq M_2(\mathcal{O}_q)^{\times}$.

Conjugation by $h_q$ gives a bijection between $U_{2,q} u_q U_{2,q}$ and $G(h_q u_q h_q^{-1})G$, where $G = \text{GU}_2(D_q)\cap \text{GL}_2(\mathcal{O}_q)$. If we fix an isomorphism as in Proposition \ref{P} then the double coset $G(h_q u_q h_q^{-1})G$ is in bijection with $\text{GSp}_4(\mathbb{Z}_q) M_q \text{GSp}_4(\mathbb{Z}_q)$ (where $M_q = \text{diag}(1,1,q,q)$).

Since by definition $h_q u_q h_q^{-1} \mapsto M_q\in \text{GSp}_4(\mathbb{Q}_q)\cap M_4(\mathbb{Z}_q)$ we see that $h_q u_q h_q^{-1} \in M_2(\mathcal{O}_q)^{\times}$ and so $u_q \in \text{GU}_2(D_q)\cap h_q^{-1} M_2(\mathcal{O}_q)^{\times} h_q$. 

However: \[h_q^{-1} M_2(\mathcal{O}_q)^{\times} h_q = g_q^{-1} (k_q^{-1} M_2(\mathcal{O}_q)^{\times} k_q) g_q = g_q^{-1} M_2(\mathcal{O}_q)^{\times} g_q,\] thus $u_q\in \text{GU}_2(D_q)\cap g_q^{-1} M_2(\mathcal{O}_q)^{\times} g_q$ and the same can be said about the $x_i$.

Also since both the conjugation and our chosen isomorphism respect similitude we find that $\mu(u_q) = \mu(M_q) = q$ and so $\mu(U_{2,q} u_q U_{2,q})\subseteq q\mathbb{Z}_q^{\times}$. In particular $\mu(x_i) \in q\mathbb{Z}_q^{\times}$.

Globally we now see that \[x_i \in \text{GU}_2(D)\cap \prod_l\left(\text{GU}_2(D_l)\cap g_l^{-1}M_2(\mathcal{O}_l)^{\times}g_l\right) = \text{GU}_2(D) \cap g^{-1}M_2(\mathcal{O})^{\times}g\] for each $i$. We also observe that $\mu(x_i)\in\mathbb{Z}\cap\left(q\mathbb{Z}_q^{\times}\prod_{l\neq q}\mathbb{Z}_l^{\times}\right) = \{\pm q\}$. However in our case the similitude is positive definite so that $\mu(x_i) = q$.

Thus the $x_i$ can be taken to lie in $Y_q$. It is clear that each such element lies in the double coset.

It remains to see which elements of $Y_q$ generate the same left coset. We have $x_i U_2 = x_j U_2$ if and only if $x_j^{-1} x_i \in U_2$. But also $x_i, x_j\in \text{GU}_2(D)$, hence $x_j^{-1} x_i\in \text{GU}_2(D) \cap U_2 = \Gamma^{(2)}$. So equivalence of left cosets is upto right multiplication by $\Gamma^{(2)}$.
\end{proof}

We have a nice formula for the degree of $T_{u,q}$, found in the work of Ihara \cite{ihara}.

\begin{prop}
For $q\neq p$ we have that deg$(T_{u,q}) = (q+1)(q^2+1)$.
\end{prop}

Employing similar arguments to Proposition \ref{R} we get the following:

\begin{prop}
Let $D$ be a quaternion algebra over $\mathbb{Q}$ ramified at $p,\infty$ for some $p\in\{2,3\}$. Suppose $u\in\text{GU}_2(D_{\mathbb{A}_f})$ is chosen as in Definition \ref{Q}. Then \[U_1 u U_1 = \coprod_{[x_i]\in (\text{GU}_2(D)_q\cap\text{M}_2(\mathcal{O})^{\times})/ \Gamma^{(1)}} x_i U_1.\]
\end{prop}

Since $\Gamma^{(1)}$ is given explicitly it is possible to write down explicit representatives in this case.

\begin{cor}\label{S}
Let $n\in\mathbb{N}$. For each $k\in\mathbb{N}$ let $X_k = \{\alpha\in\mathcal{O}\,|\,N(\alpha) = k\}$, $t_k = |X_k/\mathcal{O}^{\times}|$ and $x_{1,k}, x_{2,k}, ..., x_{t_k,k}$ be a set of representatives for $X_k/\mathcal{O}^{\times}$. For such a choice of $k$ define: \[R_{k} :=  \left\{\left(\begin{array}{cc}x_{i,k} & v\\ w & x_{j,k}\end{array}\right)\,\Bigg|\,\begin{array}{c}1\leq i,j\leq t_k,\quad v,w\in X_{n-k}\\ x_{i,k}\overline{w} + v\overline{x_{j,k}} = 0\end{array}\right\}.\]

The following matrices are representatives for $(\text{GU}_2(D)_n\cap M_2(\mathcal{O})^{\times})/ \Gamma^{(1)}$: \[\bigcup_{k=m+1}^n R_k, \qquad \text{if $n=2m+1$ is odd}\]\[\left(\bigcup_{k=m+1}^n R_k\right)\cup R_{m}',\qquad \text{if $n=2m$ is even}.\] The finite subset $R_{m}'\subset R_m$ is to be constructed in the proof.
\end{cor}

\begin{proof}
Let $\nu = \left(\begin{array}{cc}\alpha & \beta\\ \gamma & \delta\end{array}\right)\in M_2(\mathcal{O})^{\times}$. In order for $\nu\in\text{GU}_2(D)_n$ to hold we must satisfy the equations: \[N(\alpha) + N(\beta) = n\]\[N(\gamma) + N(\delta) = n\]\[\alpha\overline{\gamma} + \beta\overline{\delta} = 0.\] In a similar vein to previous discussion these equations imply that $N(\alpha) = N(\delta)$ and $N(\beta) = N(\gamma)$. Note that the first equation implies that $0\leq N(\alpha) \leq n$. 

We wish to study equivalence of these matrices under right multiplication by $\Gamma^{(1)}$. 

\textbf{Case 1:}  $N(\alpha) \neq N(\beta)$.

We may assume that $N(\alpha) > \frac{n}{2}$ since for $x,y\in\mathcal{O}^{\times}$: \[\left(\begin{array}{cc}\alpha & \beta\\ \gamma & \delta\end{array}\right)\left(\begin{array}{cc}0 & x\\ y & 0\end{array}\right) = \left(\begin{array}{cc}\beta y & \alpha x\\ \delta y & \gamma x\end{array}\right)\] and $N(\beta y) = N(\beta) = n - N(\alpha) > n - \frac{n}{2} = \frac{n}{2}$.

Under this assumption there are no anti-diagonal equivalences so it remains to check for diagonal equivalences.

Now: \[\left(\begin{array}{cc}\alpha & \beta\\ \gamma & \delta\end{array}\right)\left(\begin{array}{cc}x & 0\\ 0 & y\end{array}\right) = \left(\begin{array}{cc}\alpha x & \beta y\\ \gamma x & \delta y\end{array}\right).\]

Letting $k = N(\alpha)$ choose $x,y\in\mathcal{O}^{\times}$ so that $\alpha x = x_{i,k}$ and $\delta y = x_{j,k}$ for some $1\leq i,j\leq t_{k}$. Then $\nu$ is equivalent to $\left(\begin{array}{cc}x_{i,k} & v\\ w & x_{j,k}\end{array}\right)$. Clearly the matrices of this form are inequivalent.

It is now clear that $R_k$ gives representatives for the particular subcase $N(\alpha)=k > \frac{n}{2}$. 

\textbf{Case 2:} $N(\alpha) = N(\beta)=\frac{n}{2} = m$.

The matrices $\left(\begin{array}{cc}x_{i,m} & v\\ w & x_{j,m}\end{array}\right)$ may now have extra anti-diagonal equivalences.

Suppose \[\left(\begin{array}{cc}x_{i,m} & v\\ w & x_{j,m}\end{array}\right)\left(\begin{array}{cc}0 & x\\ y & 0\end{array}\right) = \left(\begin{array}{cc}x_{s,m} & v'\\ w' & x_{t,m}\end{array}\right).\] Then $x,y$ are uniquely determined: \[x = \frac{\overline{w}x_{t,m}}{m}\]\[y = \frac{\overline{v}x_{s,m}}{m}.\] Thus each such matrix $\left(\begin{array}{cc}x_{i,m} & v\\ w & x_{j,m}\end{array}\right)$ with $v,w\in X_m$ can only be equivalent to at most one other matrix: \[\left(\begin{array}{cc}x_{s,m} & \frac{x_{i,m}\overline{w}x_{t,m}}{m}\\ \frac{x_{j,m}\overline{v}x_{s,m}}{m} & x_{t,m}\end{array}\right),\] where $v\sim x_{s,m}$ and $w\sim x_{t,m}$ under the action of right unit multiplication.

Let $R_{m}'$ be a set consisting of a choice of matrix from each of these equivalence pairs (as $x_{i,m}$ and $x_{j,m}$ run through representatives for $X_m/\mathcal{O}^{\times}$ and $v,w$ run through elements of $X_m$ satisfying $x_{i,m}\overline{w} + v\overline{x}_{j,m}=0$). Then it is now clear that $R_{m}'$ is a set of representatives for this subcase.
\end{proof}

In the subcase $k=n-1$ it is often easier to use anti-diagonal equivalence (since $X_1/\mathcal{O}^{\times} = \{1\}$). In this case we can identify: \[R_{n-1} \longleftrightarrow \left\{\left(\begin{array}{cc}1 & z\\ -\overline{z} & 1\end{array}\right)\,\vline\,z\in X_{n-1}\right\}.\]

When $n=2$ exactly half of these will form a set of representatives. In fact it is simple to see that the equivalent pairs would be: \[\left(\begin{array}{cc}1 & z\\ -\overline{z}& 1\end{array}\right) \sim \left(\begin{array}{cc}1 & -z\\ \overline{z} & 1\end{array}\right)\] Thus: \[R_{1}' = \left\{\left(\begin{array}{cc}1 & z_i\\ -\overline{z_i} & 1\end{array}\right)\,\vline\,[z_i]\in \mathcal{O}^{\times}/\{\pm 1\}\right\}.\]

\begin{example}
If we apply Corollary \ref{S} to the choices: \[D = \left(\frac{-1,-1}{\mathbb{Q}}\right)\]\[\mathcal{O} = \mathbb{Z}\oplus\mathbb{Z}i\oplus\mathbb{Z}j\oplus\mathbb{Z}\frac{1+i+j+k}{2}\]\[n=3\]\[X_3/\mathcal{O}^{\times} = \{[1\pm i\pm j]\}\] we find that Hecke representatives for $U_1$ with ramified prime $p=2$ and $q=3$ are given by: \[\left(\begin{array}{cc}x & 0\\ 0 & y\end{array}\right), \qquad x,y\in\{1\pm i \pm j\}\]\[\left(\begin{array}{cc} 1 & z \\ -\overline{z} & 1\end{array}\right), \qquad z\in\mathcal{O}, N(z)=2.\] There are $40$ representatives here as expected and they agree with the explicit representatives given by Ibukiyama on p.$594$ of \cite{ibuk1}.\qed
\end{example}

So far we have not  needed the open compact subgroup $U_1$ but it is actually of use to us in studying $U_2$.

\begin{lem}
Let $u\in \text{GU}_2(D_{\mathbb{A}_f})$ be chosen to form the $T_{u,q}$ operator with respect to both $U_1$ and $U_2$ (for prime $q\neq p$). Then the Hecke representatives for $T_{u,q}$ with respect to $U_1$ and $U_2$ can be taken to be the same.
\end{lem}

\begin{proof}
Recall that $u$ has identity component away from $q$ and $u_q\notin U_{2,q}$ so the there is only one local condition to check, that $U_{2,q} = U_{1,q}$.

Now $U_2 = \text{Stab}_{\text{GU}_2(D_{\mathbb{A}_f})}(\mathcal{O}^2 g)$ where $g\in\text{GL}_2(D)$ is chosen so that $\mathcal{O}^2 g$ is in the non-principal genus.

We know that $U_{2,q} = \text{Stab}_{\text{GU}_2(D_q)}(\mathcal{O}_q^2 g_q)$. However by construction we know that $\mathcal{O}_q^2 g_q$ is equivalent to $\mathcal{O}_q^2$ (since $q\neq p$). Thus there exists $h_q\in \text{GU}_2(D_q)$ such that $\mathcal{O}_q^2 g_q = \mathcal{O}_q^2 h_q$.

It is then clear that: \[\text{Stab}_{\text{GU}_2(D_q)}(\mathcal{O}_q^2 g_q) = \text{Stab}_{\text{GU}_2(D_q)}(\mathcal{O}_q^2 h_q) = \text{Stab}_{\text{GU}_2(D_q)}(\mathcal{O}_q^2).\] Thus $U_{2,q} = U_{1,q}$ and so we are done.
\end{proof}

This result is useful since we have seen that it is generally easier to generate Hecke representatives for $T_{u,q}$ with respect to $U_1$.

\begin{cor}
Let the ramified prime of $D$ be $p\in\{2,3\}$. Then we may use the representatives from Corollary \ref{S} as Hecke representatives for $T_{u,q}$ with respect to $U_2$ (for $q \neq p$).
\end{cor}

\begin{proof}
Since $p\in\{2,3\}$ we know that both the class numbers of $U_1,U_2$ are $1$. Hence both admit rational Hecke representatives.

We also know that given Hecke representatives for $T_{u,q}$ with respect to $U_1$ we may use them for $U_2$. Thus the rational representatives from Corollary \ref{S} can be used for $U_2$.
\end{proof}

\subsection{Implementing the trace formula}
Now that we have algorithms that generate the data needed to use the trace formula we discuss some of the finer details in its implementation, namely how to find character values. Denote by $\chi_{j,k-3}$ the character of the representation $V_{j,k-3}$.

Given $g = \left(\begin{array}{cc}\alpha & \beta\\ \gamma & \delta\end{array}\right)\in\text{GU}_2(D)$ we may produce a matrix $A\in\text{GSp}_4(\mathbb{C})$ via the embedding: \[g \longmapsto \left(\begin{array}{cccc}\alpha_1 + \alpha_2 \sqrt{a} & \beta_1 + \beta_2\sqrt{a} & \alpha_3 + \alpha_4\sqrt{a} & \beta_3 + \beta_4\sqrt{a}\\ \gamma_1 + \gamma_2\sqrt{a} & \delta_1 + \delta_2\sqrt{a} & \gamma_3 + \gamma_4\sqrt{a} & \delta_3 + \delta_4\sqrt{a}\\ b(\alpha_3 - \alpha_4\sqrt{a}) & b(\beta_3 - \beta_4\sqrt{a}) & \alpha_1 - \alpha_2\sqrt{a} & \beta_1 - \beta_2\sqrt{a}\\ b(\gamma_3 - \gamma_4\sqrt{a}) & b(\delta_3 - \delta_4\sqrt{a}) & \gamma_1 - \gamma_2\sqrt{a} & \delta_1 - \delta_2\sqrt{a}\end{array}\right),\] where $a = i^2$ and $b=j^2$ in $D$. 

This embedding is the composition of the standard embedding $D^{\times}\hookrightarrow M_2(K(\sqrt{a}))$ and the isomorphism $\text{GU}_2(M_2(K(\sqrt{a}))) \cong \text{GSp}_4(K(\sqrt{a})) \subseteq\text{GSp}_4(\mathbb{C})$ given in Theorem \ref{L}.

We know that the image of $\text{GU}_2(\mathbb{H})_1\cap \text{GU}_2(D)$ under this embedding is a subgroup of $\text{USp}(4)$, so that the matrix $B = \frac{A}{\sqrt{\mu(A)}}\in \text{USp}(4)$. By writing $A = (\sqrt{\mu(A)}I)B$ it follows that: \[\chi_{j,k-3}(g) = \chi_{j,k-3}(A) = \mu(A)^{\frac{j+2k-6}{2}}\chi_{j,k-3}(B).\]

In order to find $\chi_{j,k-3}(B)$ we first find the eigenvalues of $B$. This is equivalent to conjugating into the maximal torus of diagonal matrices. Since $B\in\text{USp}(4)$ these eigenvalues will come in two complex conjugate pairs $z,\overline{z}, w, \overline{w}$ for $z,w$ on the unit circle.

The Weyl character formula gives: \[\chi_{j,k-3}(B) = \frac{w^{j+1}(w^{2(k-2)}-1)(z^{2(j+k-1)}-1) - z^{j+1}(z^{2(k-2)}-1)(w^{2(j+k-1)}-1)}{(z^2-1)(w^2-1)(zw-1)(z-w)(zw)^{j+k-3}}.\] For any of the cases $z^2=1, w^2=1, zw=1, z=w$ one must formally expand this concise formula into a polynomial expression (not an infinite sum since each factor on the denominator except $zw$ divides the numerator). It is easy for a computer package to compute this expansion for a given $j,k$.

\subsection{Finding the trace contribution for the new subspace}

Let $\text{tr}(T_{u,q})^{\text{new}}$ and $\text{tr}(T_{u,q})^{\text{old}}$ be the traces of the action of $T_{u,q}$ on $A_{j,k-3}^{\text{new}}(D)$ and $A_{j,k-3}^{\text{old}}(D)$ respectively. Then $\text{tr}(T_{u,q})^{\text{new}} = \text{tr}(T_{u,q}) - \text{tr}(T_{u,q})^{\text{old}}$. 

Recall that each eigenform in $\mathcal{A}_{j,k-3}^{\text{old}}(D)$ is given by a special pair of eigenforms $F_1\in\mathcal{A}(D^{\times},U,V_j)$ and $F_2\in \mathcal{A}(\text{GU}_2(D),U_2,V_{j,k-3})$. If $j>0$ then $F_1$ corresponds to a unique eigenform in $S_{j+2}^{\text{new}}(\Gamma_0(p))$ by Eichler's correspondence. Attached to the pair $(F_1,F_2)$ is an eigenform $\theta_{F_1\otimes F_2} \neq 0$ in $M_{j+2k-2}(\text{SL}_2(\mathbb{Z}))$ (it is a cusp form if $j+2k-6 \neq 0$).

Let $\alpha_n, \beta_n, \gamma_n$ be the Hecke eigenvalues of $F_1,F_2, \theta_{F_1\otimes F_2}$ respectively. Ibukiyama links the eigensystems as follows.

\begin{thm}
For $q\neq p$ we have the following identity in $\mathbb{C}(t)$: \[\sum_{k=0}^{\infty} \beta_{q^k}t^k = \frac{1 - q^{j+2k-4}t^2}{(1 - \alpha_q q^{k-2} t + q^{j+2k-3}t^2)(1 - \gamma_q t + q^{j+2k-3}t^2)}.\]
\end{thm}

\begin{cor}\label{U}
For $q\neq p$ we have $\beta_q = \gamma_q + q^{k-2}\alpha_q$. 
\end{cor}

Ibukiyama conjectures that there is a bijection between pairs of eigenforms $(F_1, \theta_F)$ and eigenforms $F_2$. With this in mind it is now possible to calculate the oldform trace contribution.

\begin{cor}\label{DD}
Suppose $j+2k-6 \neq 0$. Let $g_1, g_2, ..., g_m\in S_{j+2k-2}(SL_2(\mathbb{Z}))$ and $h_1, h_2, ..., h_n\in S_{j+2}^{\text{new}}(\Gamma_0(p))$ be bases of normalized eigenforms with Hecke eigenvalues $a_{q,g_i}$ and $a_{q,h_i}$ respectively.

Then for $q\neq p$ and $j>0$: \[\text{tr}(T_{u,q})^{\text{old}} = n\left(\sum_{i=1}^m a_{q,g_i}\right) + mq^{k-2}\left(\sum_{i=1}^n a_{q,h_i}\right).\]
\end{cor}

\section{Examples and Summary}

The following table highlights the choices for $D, \mathcal{O}, \lambda, \mu$ that were used. 
\begin{center}
\begin{tabular}{c | c | c | c | c}\hline\hline
$p$ & $D$ & $\mathcal{O}$ & $\lambda$ & $\mu$ \\ \hline $2$ & $\left(\frac{-1,-1}{\mathbb{Q}}\right)$ & $\mathbb{Z}\oplus\mathbb{Z} i\oplus\mathbb{Z} j\oplus\mathbb{Z}\left(\frac{1+i+j+k}{2}\right)$ & $1$ & $i-k$\\ $3$ & $\left(\frac{-1,-3}{\mathbb{Q}}\right)$ & $\mathbb{Z}\oplus\mathbb{Z} i\oplus\mathbb{Z} \left(\frac{1+j}{2}\right)\oplus\mathbb{Z}\left(\frac{i+k}{2}\right)$ & $1+i$ & $j$\\ $5$ & $\left(\frac{-2,-5}{\mathbb{Q}}\right)$ & $\mathbb{Z}\oplus\mathbb{Z}\left(\frac{2-i+k}{4}\right)\oplus\mathbb{Z}\left(\frac{2+3i+k}{4}\right)\oplus\mathbb{Z}\left(\frac{-1+i+j}{2}\right)$ & $2$ & $j$\\ $7$ & $\left(\frac{-1,-7}{\mathbb{Q}}\right)$ & $\mathbb{Z}\oplus\mathbb{Z} i\oplus\mathbb{Z} \left(\frac{1+j}{2}\right)\oplus\mathbb{Z}\left(\frac{i+k}{2}\right)$ & $2+\frac{1}{2}i - \frac{1}{2}k$ & $j$\\ $11$ & $\left(\frac{-1,-11}{\mathbb{Q}}\right)$ & $\mathbb{Z}\oplus\mathbb{Z} i\oplus\mathbb{Z} \left(\frac{1+j}{2}\right)\oplus\mathbb{Z}\left(\frac{i+k}{2}\right)$ & $1+3i$ & $j$\\ \hline
\end{tabular}
\end{center}

Using these choices along with the algorithms and results mentioned previously one can calculate the groups $\Gamma^{(1)},\Gamma^{(2)}$ for each such $p$, hence generating tables of dimensions of the spaces $A_{j,k-3}^{\text{new}}(D)$. These tables are given in Appendix A.$1$. 

From these tables one isolates $1$-dimensional spaces. For each possibility the MAGMA command LRatio allows us to test for large primes dividing $\Lambda_{\text{alg}}$ on the elliptic side. The cases that remained were ones where we expect to find examples of Harder's congruence. 

Tables of the congruences observed can be found in Appendix A.2. In particular for $p=2$ one observes congruences provided in Bergstr\"om \cite{bergstrom}. We finish with some new examples for $p=3$.

\begin{example}
By Appendix A.1 we see that \[\text{dim}(A_{2,5}(D)) = \text{dim}(A_{2,5}^{\text{new}}(D)) = \text{dim}(S_{2,8}^{\text{new}}(K(3)))= 1.\] Then $j=2$ and $k=8$ so that $j+2k-2 = 16$. Let $F\in S_{2,8}^{\text{new}}(K(3))$ be the unique normalized eigenform.

One easily checks that dim$(S_{16}^{\text{new}}(\Gamma_0(3))) = 2$. This space is spanned by the two normalized eigenforms with $q$-expansions: \begin{align*}f_1(\tau) &= q - 234q^2 - 2187q^3 + 21988q^4 + 280710q^5 + ... \\ f_2(\tau) &= q - 72q^2 + 2187q^3 - 27584q^4 -221490q^5 + ... \end{align*}

Indeed MAGMA informs us that $\text{ord}_{109}(\Lambda_{\text{alg}}(f_1, 10)) = 1$ and so we expect a congruence of the form: \[b_q \equiv a_q + q^9 + q^6 \bmod 109\] for all $q\neq 3$, where $b_q$ are the Hecke eigenvalues of $F$ and $a_q$ the Hecke eigenvalues of $f_1$. As discussed earlier we will only work with the case $q=2$ for simplicity.

The algorithms mentioned earlier then calculate the necessary $\frac{5760}{3^2-1} = 720$ matrices belonging to $\Gamma^{(2)}$ and the $(2+1)(2^2+1) = 15$ Hecke representatives for the operator $T_{u,2}$. Applying the trace formula we find that $\text{tr}(T_{u,2}) = -312$.

Now since $A_{2,5}(D) = A_{2,5}^{\text{new}}(D)$ we have that $\text{tr}(T_{u,2}) = \text{tr}(T_{u,2})^{\text{new}}$. Also the spaces are $1$-dimensional and so in fact $b_2 = \text{tr}(T_{u,2})^{\text{new}} = -312$. 

The congruence is then simple to check: \[-312 \equiv -234 + 2^9 + 2^6 \bmod 109.\]\qed
\end{example}

\begin{example}

We see an example where we must subtract off the oldform contribution from the trace. By Appendix A.1 we see that \[\text{dim}(A_{8,2}(D)) = 3\] whereas \[\text{dim}(A_{8,2}^{\text{new}}(D)) = \text{dim}(S_{8,5}^{\text{new}}(K(3))) = 1.\] Then $j=8$ and $k=5$ so that $j+2k-2 = 16$ again. Let $F\in S_{8,5}^{\text{new}}(K(3))$ be the unique normalized eigenform.

MAGMA informs us that $\text{ord}_{67}(\Lambda_{\text{alg}}(f_2, 13)) = 1$ and so we expect a congruence of the form: \[b_q \equiv a_q + q^{12} + q^3 \bmod 67\] for all $q\neq 3$.

Applying the trace formula this time gives $\text{tr}(T_{u,2}) = 300$. However since $\text{dim}(A_{8,2}(D)) > \text{dim}(A_{8,2}^{\text{new}}(D))$ there is an oldform contribution to this trace. In order to find it we need Hecke eigenvalues of normalized eigenforms for the spaces $S_{16}(SL_2(\mathbb{Z}))$ and $S_{10}^{\text{new}}(\Gamma_0(3))$.

It is known that dim$(S_{16}(SL_2(\mathbb{Z}))) = 1$ and that the unique normalized eigenform has $q$-expansion: \[g(\tau) = q + 216q^2 -3348q^3 + 13888q^4 + 52110+...\] Also dim$(S_{10}^{\text{new}}(\Gamma_0(3))) = 2$ and the normalized eigenforms have the following $q$-expansions: \begin{align*}h_1(\tau) &= q - 36q^2 - 81q^3 + 784q^4 - 1314q^5 + ... \\ h_2(\tau) &= q + 18q^2 + 81q^3 - 188q^4 - 1540q^5 + ...\end{align*}

Thus using Corollary \ref{DD} the oldform contribution is: \begin{align*} \text{tr}(T_{u,2})^{\text{old}} = 2a_{2,g} + 2^3(a_{2,h_1} + a_{2,h_2}) &= 512 +8(-36+18)\\ &= 288\end{align*} Hence $\text{tr}(T_{u,2})^{\text{new}} = \text{tr}(T_{u,2}) - \text{tr}(T_{u,2})^{\text{old}} = 300 - 288 = 12$. Since our space of algebraic forms is $1$-dimensional we must have $b_2 =  \text{tr}(T_{u,2})^{\text{new}} = 12$.

The congruence is then simple to check: \[12 \equiv -72 + 2^{12} + 2^3 \bmod 67.\]\qed
\end{example}

\begin{example}
Our final example is a case where the Hecke eigenvalues of the elliptic modular form lie in a quadratic extension of $\mathbb{Q}$.

By Appendix A.1 we see that \[\text{dim}(A_{6,2}(D)) = \text{dim}(A_{6,2}^{\text{new}}(D)) = \text{dim}(S_{6,5}^{\text{new}}(K(3))) = 1.\] Then $j=6$ and $k=5$ so that $j+2k-2 = 14$. Let $F\in S_{6,5}^{\text{new}}(K(3))$ be the unique normalized eigenform.

One easily checks that dim$(S_{14}^{\text{new}}(\Gamma_0(3))) = 3$. This space is spanned by the three normalized newforms with $q$-expansions: \begin{align*}f_1(\tau) &= q - 12q^2 - 729q^3 + ... \\ f_2(\tau) &= q - (27 + 3\sqrt{1969})q^2 + 729q^3 + ... \\  f_3(\tau) &= q - (27 - 3\sqrt{1969})q^2 + 729q^3+ ... \end{align*}

MAGMA informs us that $\text{ord}_{47}(N_{\mathbb{Q}(\sqrt{1969})/\mathbb{Q}}(\Lambda_\text{alg}(f_2, 11))) = 1$ and so we expect a congruence of the form: \[b_q \equiv a_q + q^{10} + q^3 \bmod \lambda\] for some prime ideal $\lambda$ of $\mathbb{Z}\left[\frac{1+\sqrt{1969}}{2}\right]$ satisfying $\lambda\mid 47$ (note that $47$ splits in this extension).

The trace formula gives $\text{tr}(T_{u,2}) = 72$ and the usual arguments show that $b_2 = 72$. It is then observed that \[N_{\mathbb{Q}(\sqrt{1969})/\mathbb{Q}}(b_2 - a_2 - 2^{10} - 2^3) = N_{\mathbb{Q}(\sqrt{1969})/\mathbb{Q}}(-933 + 3\sqrt{1969}) = 852768\] This is divisible by $47$ and so the congruence holds for $q=2$.\qed
\end{example} 

\newpage

\appendix
\section{Tables}
\subsection{Newform dimensions}

For each prime $p=2,3,5,7,11$ the following tables give values of dim$(A_{j,k}^{\text{new}}(D))$ for $0\leq j \leq 20$ even and $0\leq k \leq 15$. We use the specific quaternion algebras given in section $6$. Note that Ibukiyama conjectures that these values are equal to dim$(S_{j,k+3}^{\text{new}}(K(p)))$.

\tiny
\begin{table}[H]
\begin{tabular}{|c | c | c | c | c | c | c | c | c | c | c | c | c | c | c | c | c |} 
\hline $p=2$ & $0$ & $1$ & $2$ & $3$ & $4$ & $5$ & $6$ & $7$ & $8$ & $9$ & $10$ & $11$ & $12$ & $13$ & $14$ & $15$\\ \hline
0 & $1$ & $0$ & $0$ & $0$ & $1$ & $1$ & $1$ & $0$ & $2$ & $1$ & $2$ & $1$ & $3$ & $2$ & $3$ & $2$\\
2 & $0$ & $0$ & $0$ & $0$ & $0$ & $0$ & $0$ & $1$ & $1$ & $1$ & $1$ & $2$ & $4$ & $2$ & $4$ & $5$\\
4 & $0$ & $0$ & $0$ & $0$ & $1$ & $1$ & $1$ & $1$ & $3$ & $3$ & $4$ & $4$ & $7$ & $7$ & $9$ & $9$\\
6 & $0$ & $0$ & $0$ & $0$ & $1$ & $0$ & $2$ & $2$ & $4$ & $3$ & $5$ & $7$ & $10$ & $9$ & $13$ & $14$\\
8 & $0$ & $0$ & $0$ & $1$ & $2$ & $2$ & $2$ & $4$ & $7$ & $7$ & $9$ & $10$ & $15$ & $17$ & $20$ & $22$\\
10 & $0$ & $0$ & $0$ & $1$ & $3$ & $4$ & $4$ & $6$ & $10$ & $10$ & $14$ & $17$ & $21$ & $23$ & $29$ & $33$\\
12 & $0$ & $0$ & $1$ & $1$ & $3$ & $5$ & $6$ & $8$ & $12$ & $14$ & $17$ & $21$ & $28$ & $30$ & $37$ & $41$\\
14 & $0$ & $0$ & $1$ & $3$ & $5$ & $6$ & $9$ & $12$ & $17$ & $19$ & $24$ & $29$ & $37$ & $40$ & $49$ & $56$\\
16 & $0$ & $1$ & $2$ & $4$ & $8$ & $9$ & $13$ & $16$ & $23$ & $26$ & $32$ & $38$ & $48$ & $53$ & $63$ & $70$\\
18 & $0$ & $0$ & $2$ & $5$ & $9$ & $11$ & $15$ & $20$ & $28$ & $31$ & $39$ & $46$ & $58$ & $64$ & $76$ & $86$\\
20 & $0$ & $2$ & $3$ & $7$ & $12$ & $16$ & $20$ & $26$ & $35$ & $41$ & $50$ & $58$ & $71$ & $81$ & $94$ & $106$\\ \hline
$p=3$ & $0$ & $1$ & $2$ & $3$ & $4$ & $5$ & $6$ & $7$ & $8$ & $9$ & $10$ & $11$ & $12$ & $13$ & $14$ & $15$\\ \hline
0 & $1$ & $0$ & $0$ & $1$ & $1$ & $1$ & $2$ & $1$ & $2$ & $3$ & $3$ & $3$ & $5$ & $4$ & $5$ & $8$\\
2 & $0$ & $0$ & $0$ & $0$ & $0$ & $1$ & $1$ & $2$ & $2$ & $4$ & $4$ & $6$ & $8$ & $9$ & $11$ & $14$\\
4 & $0$ & $0$ & $0$ & $1$ & $0$ & $2$ & $3$ & $3$ & $5$ & $8$ & $8$ & $12$ & $15$ & $17$ & $22$ & $27$\\
6 & $0$ & $0$ & $1$ & $2$ & $2$ & $3$ & $7$ & $7$ & $10$ & $14$ & $16$ & $21$ & $27$ & $30$ & $37$ & $45$\\
8 & $0$ & $0$ & $1$ & $3$ & $4$ & $6$ & $8$ & $12$ & $16$ & $20$ & $25$ & $31$ & $38$ & $46$ & $54$ & $64$\\
10 & $0$ & $0$ & $1$ & $4$ & $5$ & $10$ & $13$ & $16$ & $23$ & $30$ & $35$ & $45$ & $54$ & $63$ & $76$ & $90$\\
12 & $0$ & $1$ & $4$ & $7$ & $8$ & $15$ & $20$ & $25$ & $32$ & $43$ & $49$ & $62$ & $75$ & $86$ & $102$ & $121$\\
14 & $0$ & $1$ & $5$ & $9$ & $13$ & $19$ & $27$ & $34$ & $44$ & $55$ & $67$ & $81$ & $97$ & $113$ & $133$ & $154$\\
16 & $0$ & $2$ & $6$ & $13$ & $17$ & $25$ & $36$ & $44$ & $57$ & $72$ & $84$ & $104$ & $124$ & $142$ & $167$ & $194$\\
18 & $1$ & $3$ & $10$ & $18$ & $24$ & $35$ & $47$ & $58$ & $75$ & $93$ & $109$ & $131$ & $157$ & $180$ & $209$ & $242$\\
20 & $0$ & $6$ & $12$ & $22$ & $31$ & $45$ & $58$ & $74$ & $92$ & $114$ & $136$ & $162$ & $189$ & $221$ & $254$ & $292$\\
\hline $p=5$ & $0$ & $1$ & $2$ & $3$ & $4$ & $5$ & $6$ & $7$ & $8$ & $9$ & $10$ & $11$ & $12$ & $13$ & $14$ & $15$\\ \hline
0 & $1$ & $0$ & $1$ & $1$ & $2$ & $2$ & $3$ & $3$ & $5$ & $5$ & $7$ & $8$ & $10$ & $11$ & $14$ & $16$\\
2 & $0$ & $0$ & $0$ & $0$ & $1$ & $2$ & $3$ & $5$ & $7$ & $10$ & $13$ & $17$ & $22$ & $27$ & $33$ & $40$\\
4 & $0$ & $1$ & $1$ & $3$ & $4$ & $7$ & $10$ & $14$ & $18$ & $25$ & $31$ & $39$ & $48$ & $59$ & $70$ & $84$\\
6 & $0$ & $0$ & $3$ & $4$ & $7$ & $11$ & $17$ & $22$ & $31$ & $39$ & $50$ & $63$ & $77$ & $92$ & $112$ & $131$\\
8 & $0$ & $3$ & $5$ & $9$ & $15$ & $21$ & $28$ & $40$ & $51$ & $64$ & $81$ & $99$ & $119$ & $144$ & $169$ & $198$\\
10 & $0$ & $2$ & $6$ & $12$ & $20$ & $29$ & $41$ & $54$ & $71$ & $90$ & $112$ & $136$ & $165$ & $196$ & $231$ & $270$\\
12 & $1$ & $6$ & $14$ & $22$ & $31$ & $48$ & $62$ & $81$ & $104$ & $130$ & $157$ & $193$ & $228$ & $269$ & $316$ & $366$\\
14 & $0$ & $7$ & $17$ & $27$ & $44$ & $60$ & $82$ & $107$ & $136$ & $167$ & $207$ & $247$ & $294$ & $346$ & $404$ & $465$\\
16 & $3$ & $13$ & $24$ & $43$ & $61$ & $84$ & $113$ & $145$ & $180$ & $224$ & $269$ & $322$ & $381$ & $445$ & $514$ & $594$\\
18 & $3$ & $14$ & $34$ & $53$ & $78$ & $109$ & $143$ & $181$ & $230$ & $279$ & $336$ & $401$ & $472$ & $548$ & $636$ & $727$\\
20 & $4$ & $26$ & $45$ & $72$ & $105$ & $143$ & $183$ & $236$ & $289$ & $352$ & $423$ & $500$ & $582$ & $680$ & $779$ & $ 890$\\ \hline
$p=7$ & $0$ & $1$ & $2$ & $3$ & $4$ & $5$ & $6$ & $7$ & $8$ & $9$ & $10$ & $11$ & $12$ & $13$ & $14$ & $15$\\ \hline
0 & $1$ & $1$ & $1$ & $2$ & $3$ & $4$ & $5$ & $6$ & $8$ & $10$ & $13$ & $15$ & $18$ & $22$ & $26$ & $31$\\
2 & $0$ & $0$ & $1$ & $1$ & $3$ & $5$ & $8$ & $12$ & $16$ & $22$ & $29$ & $37$ & $47$ & $57$ & $70$ & $84$\\
4 & $0$ & $1$ & $1$ & $5$ & $7$ & $12$ & $18$ & $26$ & $34$ & $47$ & $59$ & $75$ & $93$ & $114$ & $136$ & $164$\\
6 & $1$ & $3$ & $7$ & $11$ & $18$ & $26$ & $38$ & $50$ & $67$ & $85$ & $107$ & $133$ & $162$ & $194$ & $232$ & $272$\\
8 & $0$ & $6$ & $10$ & $19$ & $29$ & $43$ & $57$ & $80$ & $102$ & $130$ & $162$ & $199$ & $239$ & $289$ & $339$ & $398$\\
10 & $1$ & $5$ & $14$ & $26$ & $42$ & $60$ & $85$ & $111$ & $145$ & $183$ & $228$ & $276$ & $334$ & $396$ & $467$ & $545$\\
12 & $4$ & $15$ & $29$ & $47$ & $67$ & $98$ & $128$ & $168$ & $212$ & $265$ & $321$ & $391$ & $463$ & $546$ & $638$ & $740$\\
14 & $4$ & $18$ & $38$ & $60$ & $93$ & $127$ & $171$ & $221$ & $280$ & $344$ & $422$ & $504$ & $599$ & $703$ & $819$ & $943$\\
16 & $5$ & $27$ & $49$ & $86$ & $122$ & $170$ & $226$ & $291$ & $361$ & $449$ & $539$ & $646$ & $762$ & $892$ & $1030$ & $1189$\\
18 & $13$ & $37$ & $76$ & $116$ & $168$ & $228$ & $299$ & $377$ & $473$ & $573$ & $690$ & $818$ & $962$ & $1116$ & $1291$ & $1475$\\
20 & $13$ & $54$ & $94$ & $150$ & $214$ & $291$ & $373$ & $477$ & $585$ & $712$ & $852$ & $1008$ & $1174$ & $1367$ & $1567$ & $1791$\\ \hline
$p=11$ & $0$ & $1$ & $2$ & $3$ & $4$ & $5$ & $6$ & $7$ & $8$ & $9$ & $10$ & $11$ & $12$ & $13$ & $14$ & $15$\\ \hline
0 & $1$ & $1$ & $2$ & $3$ & $4$ & $6$ & $8$ & $11$ & $15$ & $19$ & $24$ & $31$ & $38$ & $46$ & $56$ & $67$\\
2 & $0$ & $1$ & $2$ & $4$ & $9$ & $14$ & $21$ & $31$ & $43$ & $57$ & $75$ & $95$ & $119$ & $147$ & $178$ & $213$\\
4 & $1$ & $4$ & $6$ & $15$ & $22$ & $35$ & $51$ & $71$ & $93$ & $125$ & $157$ & $197$ & $243$ & $296$ & $353$ & $422$\\
6 & $3$ & $5$ & $18$ & $27$ & $44$ & $66$ & $94$ & $124$ & $168$ & $212$ & $268$ & $332$ & $405$ & $484$ & $581$ & $681$\\
8 & $2$ & $17$ & $28$ & $49$ & $77$ & $111$ & $149$ & $205$ & $261$ & $331$ & $413$ & $506$ & $607$ & $730$ & $858$ & $1005$\\
10 & $7$ & $20$ & $43$ & $75$ & $115$ & $161$ & $225$ & $293$ & $377$ & $475$ & $586$ & $709$ & $856$ & $1012$ & $1189$ & $1386$\\
12 & $11$ & $38$ & $74$ & $120$ & $170$ & $248$ & $342$ & $422$ & $536$ & $667$ & $808$ & $983$ & $1163$ & $1372$ & $1603$ & $1857$\\
14 & $15$ & $53$ & $103$ & $159$ & $243$ & $329$ & $439$ & $567$ & $714$ & $875$ & $1072$ & $1278$ & $1515$ & $1778$ & $2068$ & $2379$\\
16 & $26$ & $78$ & $138$ & $230$ & $324$ & $444$ & $586$ & $749$ & $928$ & $1147$ & $1377$ & $1642$ & $1937$ & $2261$ & $2610$ & $3008$\\
18 & $38$ & $100$ & $198$ & $298$ & $428$ & $582$ & $759$ & $954$ & $1195$ & $1447$ & $1738$ & $2063$ & $2421$ & $2806$ & $3246$ & $3707$\\
20 & $44$ & $148$ & $252$ & $390$ & $554$ & $745$ & $954$ & $1215$ & $1487$ & $1804$ & $2157$ & $2547$ & $2966$ & $3448$ & $3951$ & $4509$\\ \hline
\end{tabular}
\end{table}

\normalsize
\newpage
\subsection{Congruences}

The following table gives information on the congruences found. For simplicity we only give the Hecke eigenvalues at $q=3$ when $p=2$ and $q=2$ when $p=3,5,7,11$.

Note that there is no congruence at level $11$ (even though one is expected). This does not contradict the conjecture since $\lambda\mid 11$ in this case.

Whenever $a_q$ is rational we give the Hecke eigenvalue explicitly. When it lies in a bigger number field we give the minimal polynomial $f(x)$ defining $\mathbb{Q}_f$ (then the Hecke eigenvalue $a_2$ in all of our cases is exactly a root $\alpha$ of this polynomial).

The large primes given are the rational primes lying below the prime for which the congruence holds. 

\small
\begin{table}[H]
\begin{tabular}{|c|c|c|c|c|c|}
\hline & $(j,k)$ & $N(\lambda)$ & tr$(T_q)$ &  $b_q$ & $a_q$\\ \hline
$p=2$&$(0,14)$ & $37$ & $2223720$ & $2223720$ & $97956$\\
&$(2,10)$ & $61$ & $18360$ & $18360$ & $-13092$\\
&$(2,11)$ & $71$& $-57528$ & $-57528$ & $59316$\\
&$(2,12)$ & $29$ & $-122040$ & $-122040$ & $-505908$\\
&$(4,10)$ & $61$ & $-189720$ & $-189720$ & $71604$\\
&$(6,7)$ & $29$ & $1872$ & $3240$ & $6084$\\
&$(10,6)$ & $109$ & $216$ & $216$ & $-13092$\\
&$(12,5)$ & $79$ & $77544$ & $-7560$ & $-53028$\\
&$(12,6)$ & $23$ & $-275688$ & $30600$ & $71604$\\
&$(14,5)$ & $379$ & $102960$ & $63000$ & $59316$\\
&$(16,4)$ & $37$ & $-97488$ & $-23400$ & $71604$\\ \hline
$p=3$ &$(2,8)$ & $109$ & $-312$ & $-312$ & $-234$\\
&$(4,6)$ & $23$ & $-36$ & $-36$ & $-12$\\
&$(6,5)$ & $47$ & $72$ & $72$ & $x^2 + 54x - 16992$\\
&$(8,5)$ & $67$ & $300$ & $12$ & $-72$\\
&$(10,5)$ & $433$ & $120$ & $24$ & $x^2 - 594x - 42912$\\
&$(12,4)$ & $23$ & $-1716$ & $132$ & $204$\\
&$(14,4)$ & $617$ & $-240$ & $72$ & $x^2 - 702x - 664128$\\ \hline
$p=5$ &$(2,7)$ & $61$ & $-76$ & $-76$ & $x^3 - 142x^2 - 11144x + 901248$\\ \hline
$p=7$ &$(2,5)$ & $263$ & $-44$ & $-44$ & $x^3 - 21x^2 - 1326x + 19080$\\ 
&$(4,4)$ & $101$ & $-2$ & $-2$ & $x^2 + 6x - 184$\\
&$(4,5)$ & $43$ &$-70$ & $10$ & $x^2 + 54x - 2640$\\ \hline
$p=11$ &$(2,4)$ & $11$ & $-20$ & $-20$ & N/A\\ \hline
\end{tabular}
\end{table}

\end{document}